\documentclass[11pt]{amsart}
\usepackage{url}

\usepackage{amsmath,amssymb,amsthm,mathtools}

\numberwithin{equation}{section}

\theoremstyle{plain}
\newtheorem{thm}{Theorem}[section]
\newtheorem{lem}[thm]{Lemma}
\newtheorem{prop}[thm]{Proposition}
\newtheorem{cor}[thm]{Corollary}

\theoremstyle{definition}
\newtheorem{defn}[thm]{Definition}
\newtheorem{ex}[thm]{Example}

\newtheorem{rem}[thm]{Remark}

\newcommand{\N}{\mathbb{N}}
\newcommand{\Z}{\mathbb{Z}}
\newcommand{\Q}{\mathbb{Q}}
\newcommand{\R}{\mathbb{R}}
\newcommand{\C}{\mathbb{C}}

\newcommand{\eps}{\varepsilon}

\newcommand{\bl}{{\boldsymbol{l}}}
\newcommand{\bp}{{\boldsymbol{p}}}
\newcommand{\bq}{{\boldsymbol{q}}}
\newcommand{\bs}{{\boldsymbol{s}}}
\newcommand{\bt}{{\boldsymbol{t}}}
\newcommand{\bi}{{\boldsymbol{i}}}
\newcommand{\bj}{{\boldsymbol{j}}}
\newcommand{\frH}{\mathfrak{H}}
\newcommand{\astb}{\mathbin{\underline{*}}}
\newcommand{\bx}{{\boldsymbol{x}}}
\newcommand{\by}{{\boldsymbol{y}}}

\DeclareMathOperator{\Hom}{Hom}
\DeclareMathOperator{\Alg}{Alg}
\newcommand{\abs}[1]{\lvert #1 \rvert}

\title{Interpolant of Truncated Multiple Zeta Functions}
\author{Kentaro Ihara \and Yayoi Nakamura \and Shuji Yamamoto}

\thanks{This research was supported by JSPS KAKENHI JP18H05233, JP18K03221 and JP21K03185. }

\address[K.~Ihara]{
    Department of Mathematics, Faculty of Science and Engineering, Kindai University, 
    3-4-1 Kowakae, Higashi-Osaka, Osaka 577-8502, Japan
}
\email{k-ihara@math.kindai.ac.jp}

\address[Y.~Nakamura]{
    Department of Mathematics, Faculty of Science and Engineering, Kindai University, 
    3-4-1 Kowakae, Higashi-Osaka, Osaka 577-8502, Japan
}
\email{yayoi@math.kindai.ac.jp}

\address[S.~Yamamoto]{
    Department of Mathematics, Faculty of Science and Technology, Keio University, 
    3-14-1 Hiyoshi, Kouhoku-ku, Yokohama 223-8522, Japan
}
\email{yamashu@math.keio.ac.jp}

\begin{document}

\begin{abstract}
We introduce an analytic function $\Psi(s_1,\ldots,s_r;w)$ that interpolates truncated 
multiple zeta functions $\zeta_N(s_1,\ldots,s_r)$. 
We represent this interpolant as a Mellin transform of a function $G(q_1,\ldots,q_r;w)$ and, 
using this expression, give the analytic continuation. 
Further, the harmonic product relations for $\Psi$ and $G$ are established via relevant Hopf algebra structures, 
and some properties of the function $G$ are provided. 
\end{abstract}

\maketitle

\section{Introduction}
For a tuple of complex variables $\bs=(s_1,\ldots,s_d)$, 
let $\zeta(\bs)$ denote the \emph{multiple zeta function} (of Euler--Zagier type) defined by 
\[\zeta(\bs)\coloneqq\sum_{j_1>\cdots>j_d>0}\frac{1}{j_1^{s_1}\cdots j_d^{s_d}}. \]
It is known that this series converges and gives a holomorphic function on the region 
\[\Re(s_1+\cdots+s_m)>m \quad (m=1,\ldots,d),\]
and this function continues meromorphically to the whole space $\C^d$ \cite{AET,Ma,Zh}. 
Moreover, the multiple zeta functions satisfy the harmonic (or stuffle) product formula, 
the simplest example of which is 
\[\zeta(s_1)\zeta(s_2)=\zeta(s_1,s_2)+\zeta(s_2,s_1)+\zeta(s_1+s_2). \]
In general, the harmonic product formula expresses the product of two multiple zeta functions 
as a sum of multiple zeta functions (see \S\ref{sec:HP} for the precise formulation). 
In fact, for any $N\in\N_0$, the same formula holds for the truncated sums 
\[\zeta_N(\bs)\coloneqq\sum_{N\ge j_1>\cdots>j_d>0}\frac{1}{j_1^{s_1}\cdots j_d^{s_d}}\]
(here and throughout this paper, $\N$ stands for the set of positive integers and $\N_0=\N\cup\{0\}$). 
In this paper, we investigate the analytic function $\Psi(\bs;w)$ which interpolates these truncated sums, 
i.e., satisfies $\Psi(\bs;N)=\zeta_N(\bs)$ for all $N\in\N_0$. 

For positive integer points $\bl=(l_1,\ldots,l_d)\in\N^d$, 
several authors have studied analytic functions interpolating $\zeta_N(\bl)$ for $N\in\N_0$.  
By studying such functions, one can deduce interesting consequences on the multiple zeta \emph{values} 
(that is, the values $\zeta(\bl)$ at $\bl\in\N^d$ with $l_1\ge 2$). 
For example, Kawashima \cite{Ka} and Zlobin \cite{Zl} had introduced 
the interpolating functions of the truncated multiple zeta-star values 
\[\zeta^\star_N(\bl)\coloneqq\sum_{N\ge j_1\ge\cdots\ge j_d>0}\frac{1}{j_1^{l_1}\cdots j_d^{l_d}}\]
to imply $\Q$-linear relations among multiple zeta values. 
The main result in \cite{Ka} provides a wide class of relations, called the Kawashima relation, 
which includes some of fundamental linear relations such as 
the duality relation \cite[\S9]{Za} and, more generally, the Ohno relation \cite[Theorem 1]{O}. 

The second author and Kusunoki \cite{KN} had constructed a function $\Psi(\bl;w)$ 
interpolating the multiple harmonic sums $\zeta_N(\bl)$. 
This is defined inductively by 
\begin{equation}\label{eq:Psi_l}
\Psi(l_1,\ldots,l_d;w)\coloneqq \sum_{j=1}^\infty
\Biggl\{\frac{\Psi(l_2,\ldots,l_d;j-1)}{j^{l_1}}-\frac{\Psi(l_2,\ldots,l_d;w+j-1)}{(w+j)^{l_1}}\Biggr\} 
\end{equation}
with the initial value $\Psi(\varnothing;w)\coloneqq 1$ for $d=0$. 
When $d=1$, $\Psi(l;w)$ reduces to the classical polygamma function. 
For $d\ge 1$, $\Psi(\bl;w)$ is a meromorphic function on $w\in\C$ with poles at $w=-1,-2,-3,\cdots$ 
and with obvious zeros at $w=0,1,\ldots,d-1$. 
By the interpolation property $\Psi(\bl;N)=\zeta_N(\bl)$ for $N\in\N_0$, we clearly have 
\begin{equation}\label{eq:lim Psi_l(N)}
\lim_{N\to\infty} \Psi(\bl;N)=\zeta(\bl)
\end{equation}
for admissible $\bl$ (i.e., $l_1\ge 2$). 
In fact, it holds that $\lim_{w\to\infty}\Psi(\bl;w)=\zeta(\bl)$ if $\lvert\arg w\rvert<\pi$. 
See \cite{KNS,KN} for details and applications of $\Psi(\bl;w)$. 

In this paper, we study a function $\Psi(\bs;w)$ of 
\emph{complex variables} $\bs=(s_1,\ldots,s_d)$ and $w$, which extends the above $\Psi(\bl;w)$ 
and satisfies the interpolation property $\Psi(\bs;N)=\zeta_N(\bs)$ for $N\in\N_0$. 
The content of each section is as follows: 
In \S2, the function $\Psi(\bs;w)$ is introduced and an expression in terms of 
the multiple zeta(-star) function is given (Proposition \ref{prop:Psi Hurwitz}). 
In \S3, a holomorphic function $G(\bq;w)$ on a certain region $U_d\times\C$ is introduced. 
We show some properties of $G(\bq;w)$, which will be used later. 
In \S4, we prove the main result of this paper, namely, the Mellin integral expression of $\Psi(\bs;w)$: 
\[
\Psi(\bs;w)=\frac{1}{\prod_{m=1}^d\Gamma(s_m)}
\int_0^1\cdots\int_0^1 G(\bq;w)\prod_{m=1}^d(-\log q_m)^{s_m-1}\,dq_1\cdots dq_d 
\]
(see Theorem \ref{thm:Psi integral} for the precise statement). 
As a corollary, the analytic continuation of $\Psi(\bs;w)$ is obtained. 
In \S5, the harmonic product formulas for $\Psi(\bs;w)$ and $G(\bq;w)$ are discussed. 
By using certain Hopf algebras, we establish these formulas in a parallel way. 
In \S6, we show miscellaneous results on the function $G(\bq;w)$. 

Some of the results in this paper were obtained in the years starting in 2017 
as a joint work by the 1st and 2nd authors. 
When the 2nd author gave a talk on these results at RIMS in 2019, 
Komori gave comments and papers \cite{Ko, ko1} on his recent research. 
It turned out that he also independently introduced an interpolation function 
for multiple zeta functions and proved a special case 
of the integral representation we gave in this paper. 

\subsection{Notation}\label{subsec:Notation}

Let $X$ be a set.  
Throughout the paper, we use the following notation: 
For $x_0\in X$, $\bx=(x_1,\ldots,x_d)\in X^d$, 
$\bx'=(x'_1,\ldots,x'_{d'})\in X^{d'}$, and $0\le m\le d$, 
\begin{align*}
(x_0,\bx)&\coloneqq (x_0,x_1,\ldots,x_d)\in X^{d+1}, \\
(\bx,\bx')&\coloneqq (x_1,\ldots,x_d,x'_1,\ldots,x'_{d'})\in X^{d+d'}, \\
\bx_{[m]}&\coloneqq (x_1,\ldots,x_m)\in X^m, \\
\bx^{[m]}&\coloneqq (x_{m+1},\ldots,x_d)\in X^{d-m}. 
\end{align*}

In particular, we have $\bx_{[0]}=\bx^{[d]}=\varnothing$, 
where $\varnothing$ denotes the empty tuple. 
In typical cases, we will use the notation $\bs=(s_1,\ldots,s_d)\in\C^d$ for $X=\C$ and 
$\bq=(q_1,\ldots,q_d)\in\C^d$ for $X=\C$ or $X=\R_{>0}$

For a complex number $s$, the power function $w^s$ is considered on the cut complex plane 
$\C\setminus\R_{\le 0}$ by using the branch of logarithm with $-\pi<\Im\log w<\pi$.

\section{Interpolant $\Psi(\bs;w)$}

For integers $d\ge 0$, we define the function $\Psi^d(\bs;w)$ of $d+1$ complex variables 
$\bs=(s_1,\ldots,s_d)$ and $w$ inductively by 
\begin{align}
\Psi^0(\varnothing;w)&\coloneqq 1,\notag\\
\Psi^d(\bs;w)&\coloneqq\sum_{j=1}^\infty
\biggl\{\frac{\Psi^{d-1}(\bs^{[1]};j-1)}{j^{s_1}}-\frac{\Psi^{d-1}(\bs^{[1]};w+j-1)}{(w+j)^{s_1}}\biggr\} 
\qquad(d>0). \label{eq:Psi(s;w) defn}
\end{align}
We let the variable $w$ vary in the domain $\C\setminus\R_{\le -1}$, 
so that the powers $(w+j)^s$ are well-defined following the convention in \S\ref{subsec:Notation}. 
In the following, we often omit the superscript $d$ from $\Psi^d$. 

We also consider the multiple zeta-star function of Hurwitz type 
\[\zeta^\star(\bs;w)\coloneqq \sum_{j_1\ge\cdots\ge j_d>0}
\frac{1}{(w+j_1)^{s_1}\cdots(w+j_d)^{s_d}}, \]
which absolutely converges (at least) if $\Re(s_1),\ldots,\Re(s_d)>1$ and $w\notin\R_{\le -1}$. 
For $d=0$, we set $\zeta^\star(\varnothing;w)\coloneqq 1$. 

\begin{prop}\label{prop:Psi Hurwitz}
If $\Re(s_1),\ldots,\Re(s_d)>1$ and $w\notin\R_{\le -1}$, 
the series \eqref{eq:Psi(s;w) defn} absolutely converges and the identity 
\begin{equation}\label{eq:Psi Hurwitz}
\Psi(\bs;w)=\sum_{i=0}^d(-1)^i\zeta^\star(s_i,\ldots,s_1;w)\,\zeta(s_{i+1},\ldots,s_d) 
\end{equation}
holds. Moreover, for any $N\in\N_0$, we have the interpolation property 
\[\Psi(\bs;N)=\zeta_N(\bs). \]
\end{prop}
\begin{proof}
We proceed by induction on $d\ge 1$. 
For $d=1$, the convergence of \eqref{eq:Psi(s;w) defn} and the identity \eqref{eq:Psi Hurwitz} are obvious. 
Moreover, for $w=N\in\N_0$, 
\eqref{eq:Psi(s;w) defn} becomes the telescoping series 
\[\Psi(s;N)=\sum_{j=1}^\infty\biggl\{\frac{1}{j^s}-\frac{1}{(N+j)^s}\biggr\}
=\sum_{j=1}^N\frac{1}{j^s}=\zeta_N(s).\]
For $d\ge 2$, we use the induction hypothesis to compute as 
\begin{align*}
\Psi(\bs;w)&=\sum_{j=1}^\infty\Biggl\{\frac{\zeta_{j-1}(s_2,\ldots,s_d)}{j^{s_1}}\\
&\qquad\qquad -\frac{1}{(w+j)^{s_1}}\sum_{i=1}^d (-1)^{i-1}\zeta^\star(s_i,\ldots,s_2;w+j-1)
\zeta(s_{i+1},\ldots,s_d)\Biggr\}\\
&=\zeta(s_1,s_2,\ldots,s_d)
+\sum_{i=1}^d (-1)^i \zeta^\star(s_i,\ldots,s_1;w)\zeta(s_{i+1},\ldots,s_d), 
\end{align*}
and obtain the first statement. For $N\in\N_0$, we have 
\[\Psi(\bs;N)=\sum_{j=1}^N\frac{\Psi(s_2,\ldots,s_d;j-1)}{j^{s_1}}
=\sum_{j=1}^N\frac{\zeta_{j-1}(s_2,\ldots,s_d)}{j^{s_1}}
=\zeta_N(s_1,s_2,\ldots,s_d). \qedhere \] 
\end{proof}

\section{Integrand $G(\bq;w)$}
In this section, we introduce and study a function $G^d(\bq;w)$, 
which will be the main factor in the integral represention of $\Psi(\bs;w)$. 

For integers $d\ge 0$, let $U_d$ denote the domain in $\C^d$ defined by 
\[U_d\coloneqq \bigl\{(q_1,\ldots,q_d)\in(\C^\times)^d\bigm|
-\pi<\arg q_i+\cdots+\arg q_j<\pi\ (1\le\forall i\le\forall j\le d)\bigr\}. \]
Note that the power functions $(q_i\cdots q_j)^w$ on $U_d$ are defined by the convention in \S\ref{subsec:Notation}, 
and the identity $(q_i\cdots q_j)^w=q_i^w\cdots q_j^w$ holds for any $i\le j$ and $w\in\C$. 

\begin{defn}
We define the function $G^d(\bq;w)$ on $ U_d\times\C$ inductively by 
\[G^0(\varnothing;w)\coloneqq 1,\quad 
G^1(q;w)\coloneqq\begin{cases}
\frac{1-q^w}{1-q} & (q\ne 1),\\
w & (q=1),
\end{cases}\]
and 
\begin{equation}\label{eq:G^d}
G^d(\bq;w)\coloneqq\begin{cases}
\frac{1}{1-q_d}\bigl\{G^{d-1}(\bq_{[d-1]};w)-G^{d-1}(\bq_{[d-2]},q_{d-1}q_d;w)\bigr\} & (q_d\ne 1),\\[2mm]
q_{d-1}\frac{\partial}{\partial q_{d-1}}G^{d-1}(\bq_{[d-1]};w) & (q_d=1)
\end{cases}
\end{equation}
for $d\ge 2$. Again we will often omit the superscript $d$. 
\end{defn}

\begin{prop}\label{prop3.2}
$G^d(\bq;w)$ is a holomorphic function on $U_d\times \C$. 
\end{prop}
\begin{proof}
The statement holds trivially for $d=0$, and also for $d=1$; 
at points of the form $(1,w)$, the holomorphy of $G^1$ with respect to $q$ follows from 
the following elementary fact: for a holomorphic function $f(q)$ around a point $q=a$, the function 
\[g(q)\coloneqq \begin{cases}
\frac{f(q)-f(a)}{q-a} & (q\ne a),\\[2mm]
f'(a) & (q=a)
\end{cases}\]
is also holomorphic around $q=a$. Indeed, if $f(q)=\sum_{n\ge 0}c_n(q-a)^n$, 
then one has $g(q)=\sum_{n\ge 1}c_n(q-a)^{n-1}$. 

Let $d\ge 2$ and assume that $G^{d-1}(\bq_{[d-1]};w)$ is holomorphic on $U_{d-1}\times\C$. 
Then it is obvious from the definition that $G^d(\bq;w)$ is holomorphic where $q_d\ne 1$. 
At a point $(\bq,w)$ with $q_d=1$, it is also immediate from the definition that 
$G^d(\bq;w)$ is holomorphic with respect to the variables $q_1,\ldots,q_{d-1},w$, 
and the holomorphy with respect to $q_d$ is shown as in the case of $d=1$. 
\end{proof}

When $d=2$, we have 
\begin{align*}
G^2(q_1,q_2;w)&=\frac{1}{1-q_2}\biggl\{\frac{1-q_1^w}{1-q_1}-\frac{1-(q_1q_2)^w}{1-q_1q_2}\biggr\}\\
&=\frac{q_1}{(1-q_1)(1-q_1q_2)}-\frac{q_1^w}{(1-q_1)(1-q_2)}+\frac{(q_1q_2)^w}{(1-q_1q_2)(1-q_2)}
\end{align*}
generically, i.e., where none of $q_1$, $q_2$, $q_1q_2$ is equal to $1$. 
For general $d$, we have the following expression. 

\begin{prop}\label{prop:G^d explicit}
For integers $i$ and $j$ with $1\le i\le j\le d$, set 
\[D_{ij}\coloneqq\bigl\{\bq=(q_1,\ldots,q_d)\in \C^d\bigm|q_i\cdots q_j=1\bigr\}\]
and $D\coloneqq \bigcup_{i\le j}D_{ij}$. Then we have 
\begin{equation}\label{eq:G^d explicit}
G^d(\bq;w)=\sum_{l=0}^d(-1)^l\frac{(q_1\cdots q_l)^w}{q_{l+1}\cdots q_d}
\prod_{1\le k\le l}\frac{1}{1-q_k\cdots q_l}\prod_{l<k\le d}\frac{q_{l+1}\cdots q_k}{1-q_{l+1}\cdots q_k}
\end{equation}
for $\bq\in U_d\setminus D$. 
\end{prop}
\begin{proof}
Write $\tilde{G}^d(\bq;w)$ for the right hand side. 
Since $G^d(\bq;w)$ and $\tilde{G}^d(\bq;w)$ are both meromorphic on $ U_d\times \C$, 
it suffices to show the equality assuming $\abs{q_1},\ldots,\abs{q_d}<1$.
In this range, $\tilde{G}^d(\bq;w)$ is expanded to the power series as  
\[\tilde{G}^d(\bq;w)=\sum_{l=0}^d(-1)^l(q_1\cdots q_l)^w
\sum_{\substack{0\le j_1\le\cdots\le j_l\\ j_{l+1}>\cdots>j_d\ge 0}}q_1^{j_1}\cdots q_d^{j_d}. \]
We write $A_l$ for the $l$-th term in the right hand side. 
For $l=0,\ldots,d-2$, we have 
\begin{align*}
(1-q_d)A_l&=(-1)^l(q_1\cdots q_l)^w
\sum_{\substack{0\le j_1\le\cdots\le j_l\\ j_{l+1}>\cdots>j_{d-1}\ge 0}}
q_1^{j_1}\cdots q_{d-1}^{j_{d-1}}
\cdot (1-q_d)\sum_{j_d=0}^{j_{d-1}-1}q_d^{j_d}\\
&=(-1)^l(q_1\cdots q_l)^w
\sum_{\substack{0\le j_1\le\cdots\le j_l\\ j_{l+1}>\cdots>j_{d-1}\ge 0}}
q_1^{j_1}\cdots q_{d-1}^{j_{d-1}}(1-q_d^{j_{d-1}}). 
\end{align*}
We also compute the terms for $l=d-1$ and $d$ together to obtain 
\begin{align*}
&(1-q_d)(A_{d-1}+A_d)\\
&=(-1)^{d-1}(q_1\cdots q_{d-1})^w\sum_{0\le j_1\le\cdots\le j_{d-1}}
q_1^{j_1}\cdots q_{d-1}^{j_{d-1}}\cdot
(1-q_d)\Biggl\{\sum_{j_d=0}^\infty q_d^{j_d}-q_d^w\sum_{j_d=j_{d-1}}^\infty q_d^{j_d}\Biggr\}\\
&=(-1)^{d-1}(q_1\cdots q_{d-1})^w\sum_{0\le j_1\le\cdots\le j_{d-1}}
q_1^{j_1}\cdots q_{d-1}^{j_{d-1}}(1-q_d^{w+j_{d-1}}). 
\end{align*}
By combining these computations, we see that 
\[(1-q_d)\tilde{G}^d(\bq;w)=\tilde{G}^{d-1}(\bq_{[d-1]};w)-\tilde{G}^{d-1}(\bq_{[d-2]},q_{d-1}q_d;w), \]
that is, $\tilde{G}^d(\bq;w)$ satisfies the same recurrence relation as $G^d(\bq;w)$. 
Hence the desired identity follows by induction on $d$. 
\end{proof}

As a bi-product of the above proof, we obtain the following series expansion: 

\begin{cor}
For $\bq=(q_1,\ldots,q_d)\in U_d$ with $\abs{q_1},\ldots,\abs{q_d}<1$, we have 
\begin{align}
\label{eq:G^d series}
G(\bq;w)&=\sum_{l=0}^d (-1)^l (q_1\cdots q_l)^w
\sum_{\substack{0\le j_1\le\cdots\le j_l\\ j_{l+1}>\cdots>j_d\ge 0}}q_1^{j_1}\cdots q_d^{j_d}\\
\notag 
&=\sum_{l=1}^d (-1)^{l-1} (q_1\cdots q_{l-1})^w(1-q_l^w)
\sum_{0\le j_1\le\cdots\le j_l>j_{l+1}>\cdots>j_d\ge 0}q_1^{j_1}\cdots q_d^{j_d}. 
\end{align}
\end{cor}
\begin{proof}
The first expression is given in the proof of Proposition \ref{prop:G^d explicit}. 
The second expression follows from the decomposition of the nested sum  
\[\sum_{\substack{0\le j_1\le\cdots\le j_l\\ j_{l+1}>\cdots>j_d\ge 0}}
=\sum_{0\le j_1\le\cdots\le j_l>j_{l+1}>\cdots>j_d\ge 0}
+\sum_{0\le j_1\le\cdots\le j_l\le j_{l+1}>\cdots>j_d\ge 0}\]
for $l=0,\ldots,d$. 
\end{proof}

We will need the following lemma, which gives an estimate of $G(\bq;w)$ when some entry of $\bq$ tends to zero, 
in the next section. Let $\theta$ be a number with $0<\theta<\pi/d$ and set 
\[C_\theta\coloneqq \{q\in\C^\times\mid -\theta\le\arg q\le\theta,\ \abs{q}\le 2\}. \]
Note that $C_\theta^d\subset U_d$. 

\begin{lem}\label{lem:G^d bound}
Let $\sigma$ be a positive real number. 
Then $(q_1\cdots q_d)^\sigma G(\bq;w)$ is bounded on $\bq=(q_1,\ldots,q_d)\in C_\theta^d$ 
uniformly for $w$ in any compact subset of the region $\Re(w)>-\sigma$. 
\end{lem}
\begin{proof}
We define a function $f(\bq;w)$ of $w\in\C$ with $\Re(w)>-\sigma$ and 
$\bq=(q_1,\ldots,q_d)\in \bigl(\{0\}\cup C_\theta\bigr)^d$ by 
\[f(\bq;w)\coloneqq \begin{cases}
(q_1\cdots q_d)^\sigma G^d(\bq;w) & (q_1\cdots q_d\ne 0),\\
0 & (q_1\cdots q_d=0).
\end{cases}\]
Then it suffices to prove that $f(\bq;w)$ is a continuous function. 
Since it is even holomorphic on $U_d\times\C$, it remains to show the continuity at 
each point $(\bq_0,w_0)=(q_{01},\ldots,q_{0d},w_0)$ with $q_{01}\cdots q_{0d}=0$.  
Put 
\[I\coloneqq\{i \mid q_{0i}=0\},\quad 
J\coloneqq\{j \mid q_{0j}\ne 0\}=\{1,\ldots,d\}\setminus I. \]
Then $I\ne\emptyset$ by assumption. 

For $\delta>0$ and $r_j>0$ ($j\in J$), define a subset $S=S_{\delta,(r_j)}$ of $\C^d$ by 
\[S\coloneqq\biggl\{ (q_1,\ldots,q_d)\in\C^d\biggm|
\begin{array}{l}
-\theta\le\arg q_i\le \theta,\  \abs{q_i}<\delta\ (i\in I),\\
\abs{q_j-q_{0j}}=r_j\ (j\in J)
\end{array}\biggr\}. \]
If $r_j$ are sufficiently small, we have $S\subset U_d$. 
In what follows, we prove the following claim: we can choose $r_j>0$ for $j\in J$ so that 
there is no solution of the equation $\prod_{j\in J'}q_j=1$ with $\emptyset\ne J'\subset J$ and $\bq=(q_1,\ldots,q_d)\in S$. 
Then, in particular, we have $S\subset U_d\setminus D$ for any sufficiently small $\delta>0$. 

To prove the above claim, fix a numbering $J=\{j_1,\ldots,j_n\}$ of the elements of $J$ (hence $n=\#J$). 
Then, for $m=1,\ldots,n$, one can inductively choose $r_{j_m}>0$ so that 
there is no solution of 
\[\prod_{j\in J'}q_j\cdot q_{j_m}\cdot \prod_{j\in J''}q_{0j}=1\]
with $J'\subset\{j_1,\ldots,j_{m-1}\}$, $J''\subset\{j_{m+1},\ldots,j_n\}$ 
and $\abs{q_j-q_{0j}}=r_j$ for $j\in J'$ and $j=j_m$. 
Indeed, in the $m$-th step, one can avoid the solutions of the form $q_{0j_m}\prod_{j\in J''}q_{0j}=1$ 
by replacing $q_{0j_m}$ with $q_{j_m}$, not producing other solutions by choosing sufficiently small $r_{j_m}>0$. 

We return to the proof of Lemma \ref{lem:G^d bound}. 
Take an arbitrary $\eps>0$. Let $\sigma'$ be a real number with $\Re(w_0)>-\sigma'>-\sigma$.  
Then, by the expression 
\[f(\bq;w)=(q_1\cdots q_d)^\sigma\sum_{l=0}^d(-1)^l\frac{(q_1\cdots q_l)^w}{q_{l+1}\cdots q_d}
\prod_{1\le k\le l}\frac{1}{1-q_k\cdots q_l}\prod_{l<k\le d}\frac{q_{l+1}\cdots q_k}{1-q_{l+1}\cdots q_k}\]
from Proposition \ref{prop:G^d explicit} and by the assumption that $I\ne\emptyset$, 
we see that, for a sufficiently small $\delta>0$, 
$\abs{f(\bq;w)}<\eps$ holds for all $w$ with $\Re(w)>-\sigma'$ and $\bq\in S$ ($\subset U_d\setminus D$). 
If we put 
\[W=\biggl\{(q_1,\ldots,q_d)\in\C^d\biggm|
\begin{array}{l}
-\theta\le\arg q_i\le \theta,\  \abs{q_i}<\delta\ (i\in I),\\
\abs{q_j-q_{0j}}<r_j/2\ (j\in J)
\end{array}\biggr\}\subset U_d, \]
Cauchy's integral formula shows that 
\[f(\bq;w)=\frac{1}{(2\pi i)^n}\int_{\substack{\abs{p_j-q_{0j}}=r_j\\(j\in J)}}
f(\bp;w)\prod_{j\in J}\frac{dp_j}{p_j-q_j}\]
holds for $\bq=(q_1,\ldots,q_d)\in W$, where the integral is taken over $\bp=(p_1,\ldots,p_d)$ 
such that $p_i=q_i$ for $i\in I$ and $\abs{p_j-q_{0j}}=r_j$ for $j\in J$. 
In particular, $\bp$ varies on $S$, hence the estimates $\abs{f(\bp;w)}<\eps$ and 
\[\abs{p_j-q_j}\ge\abs{p_j-q_{0j}}-\abs{q_j-q_{0j}}>r_j/2 \qquad (j\in J)\] 
imply 
\[\abs{f(\bq;w)}\le \frac{\eps}{(2\pi )^n}\prod_{j\in J}\frac{2\pi r_j}{r_j/2}=2^n\eps\]
for all $w$ with $\Re(w)>-\sigma'$ and $\bq\in W$. 
This proves the continuity of $f(\bq;w)$ at $(\bq_0,w_0)$ as desired. 
\end{proof}

\section{Integral representation of $\Psi(\bs;w)$}
In this section, we provide the following integral representation of $\Psi(\bs;w)$. 

\begin{thm}\label{thm:Psi integral}
For $\bs=(s_1,\ldots,s_d)\in\C^d$ and $w\in\C$
such that $\Re(s_1),\ldots,\Re(s_d)>1$ and $\Re(w)>-1$, 
the function $\Psi(\bs;w)$ is represented by the integral 
\begin{align}
\label{eq:Psi integral}
\Psi&(\bs;w)=\frac{1}{\prod_{m=1}^d\Gamma(s_m)}
\int_0^1\cdots\int_0^1 G(\bq;w)\prod_{m=1}^d(-\log q_m)^{s_m-1}\,dq_1\cdots dq_d\\
\notag 
&=\frac{1}{\prod_{m=1}^d\Gamma(s_m)}
\int_0^\infty\cdots\int_0^\infty G(e^{-x_1},\ldots,e^{-x_d};w)
\prod_{m=1}^d(x_m^{s_m-1}e^{-x_m})\,dx_1\cdots dx_d.
\end{align}
The integral of the right hand side converges absolutely 
for $\Re(s_1),\ldots,\Re(s_d)>0$ and $\Re(w)>-1$, uniformly on any compact set, 
and hence gives an analytic continuation of $\Psi(\bs;w)$ to that region. 
\end{thm}
\begin{proof}
We first prove the statement on the convergence of the integral by using Lemma \ref{lem:G^d bound}.  
For any $0<\sigma<1$ and $w\in\C$ with $\Re(w)>-\sigma$, we can take a constant $C>0$ such that 
$\abs{G^d(\bq;w)}\le C\abs{q_1\cdots q_d}^{-\sigma}$ for all $\bq\in (0,1]^d$. 
Then we have 
\[\Biggl| G(\bq;w)\prod_{m=1}^d(-\log q_m)^{s_m-1}\Biggr|
\le C\prod_{m=1}^d\Bigl(q_m^{-\sigma}(-\log q_m)^{\Re(s_m)-1}\Bigr),\]
which implies the desired integrability. 

Next we prove the identity \eqref{eq:Psi integral}. 
If $\Re(s_1),\ldots,\Re(s_d)>1$, we may substitute the series experssion \eqref{eq:G^d series} 
into the integral of \eqref{eq:Psi integral} and exchange the order of integral and summation: 
\begin{align*}
&\int_0^1\cdots\int_0^1 G(\bq;w)\prod_{m=1}^d(-\log q_m)^{s_m-1}dq_1\cdots dq_d\\
&=\sum_{l=0}^d(-1)^l\sum_{\substack{0\le j_1\le\cdots\le j_l\\ j_{l+1}>\cdots>j_d\ge 0}}
\prod_{m=1}^l \int_0^1 q_m^{w+j_m}\,(-\log q_m)^{s_m-1}dq_m\\
&\hspace{150pt}\times\prod_{m=l+1}^d \int_0^1 q_m^{j_m}\,(-\log q_m)^{s_m-1}dq_m\\
&=\sum_{l=0}^d(-1)^l\sum_{\substack{0\le j_1\le\cdots\le j_l\\ j_{l+1}>\cdots>j_d\ge 0}}
\prod_{m=1}^l \frac{\Gamma(s_m)}{(w+j_m+1)^{s_m}}\times\prod_{m=l+1}^d \frac{\Gamma(s_m)}{(j_m+1)^{s_m}}\\
&=\prod_{m=1}^d\Gamma(s_m)\sum_{l=0}^d(-1)^l\zeta^\star(s_l,\ldots,s_1;w)\,\zeta(s_{l+1},\ldots,s_d)
=\prod_{m=1}^d\Gamma(s_m)\cdot\Psi(\bs;w), 
\end{align*}
where the last equality is the identity \eqref{eq:Psi Hurwitz}. 
Thus \eqref{eq:Psi integral} holds. 
\end{proof}

\begin{cor}
$\Psi^d(\bs;w)$ has a holomorphic continuation to the domain $\C^d\times(\C\setminus\R_{\le -1})$.  
\end{cor}
\begin{proof}
Let us consider the second integral expression in \eqref{eq:Psi integral} 
for $(\bs;w)$ such that $\Re(s_i)>0$ and $\Re(w)>-1$. 
Just as in the well-known proof of the analytic continuation of the Riemann zeta function, 
we change the domain of integration to the $d$-fold product of the Hankel contour $\gamma$. 
This is possible by Lemma \ref{lem:G^d bound}, and thus we obtain 
\begin{equation}\label{eq:Psi Hankel}
\Psi(\bs;w)=\prod_{m=1}^d\frac{\Gamma(1-s_m)}{2\pi i\,e^{\pi is_m}}
\int_\gamma\cdots\int_\gamma 
G(e^{-x_1},\ldots,e^{-x_d};w)\prod_{m=1}^d(x_m^{s_m-1}e^{-x_m})\,dx_1\cdots dx_d.
\end{equation}
This expression gives the holomorphic continuation to $(\bs;w)\in\C^d\times\{\Re(w)>-1\}$. 

To obtain the continuation with respect to $w$, we proceed by induction on $d$. 
There is no problem for $d=0$. 
When $d>0$, the definition \eqref{eq:Psi(s;w) defn} implies the difference equation 
\begin{equation}\label{eq:Psi difference}
\Psi^d(\bs;w+1)-\Psi^d(\bs;w)=\frac{\Psi^{d-1}(\bs^{[1]};w)}{(w+1)^{s_1}}. 
\end{equation}
By using this equation and the induction hypothesis on $\Psi^{d-1}$, 
we can extend the domain of definition of $\Psi^d$ recursively. This completes the proof. 
\end{proof}

\begin{rem}
From the above argument, one can also see that $\Psi^d(\bs;w)$ can be continued 
to a multi-valued holomorphic function on $\C^d\times(\C\setminus\Z_{\le -1})$. 
Moreover, for $\bl\in\Z^d$, the monodromy around negative integers become trivial and 
$\Psi^d(\bl;w)$ is a single-valued meromorphic function with poles at $w=-1,-2,\ldots$. 
This is the interpolating function studied by Kusunoki and the second author \cite{KN}. 
\end{rem}

\section{Harmonic product}\label{sec:HP}
In this section, we discuss the \emph{harmonic product rule}, or the \emph{harmonic relation}, 
for the functions $\Psi^d(\bs;w)$, the simplest nontrivial example of which is 
\begin{equation}\label{eq:HP Psi example}
\Psi(t;w)\Psi(s;w)=\Psi(t,s;w)+\Psi(s,t;w)+\Psi(t+s;w). 
\end{equation}
There is the corresponding harmonic relation for the integrands $G(\bq;w)$. For example, 
the above relation \eqref{eq:HP Psi example} corresponds to 
\begin{equation}\label{eq:HP G example}
G(p;w)G(q;w)=G(p,q;w)+G(q,p;w)+G(pq;w), 
\end{equation}
which can be verified by explicitly computing both sides. 
Note that \eqref{eq:HP G example} implies \eqref{eq:HP Psi example} 
through the integral representation \eqref{eq:Psi integral}; 
the term $G(pq;w)$ is integrated via the change of variables $(p,q)=(r^x,r^{1-x})$ as 
\begin{align*}
&\frac{1}{\Gamma(t)\Gamma(s)}\iint_0^1 G(pq;w)(-\log p)^{t-1}(-\log q)^{s-1}dp\,dq\\
&=\frac{1}{\Gamma(t)\Gamma(s)}\iint_0^1 G(r;w)(-x\log r)^{t-1}(-(1-x)\log r)^{s-1}(-\log r)dxdr\\
&=\frac{1}{\Gamma(t)\Gamma(s)}B(t,s) \int_0^1 G(r;w)(-\log r)^{t+s-1}dr\\
&=\Psi(t+s;w). 
\end{align*}
The general harmonic relations for $\Psi(\bs;w)$ and $G(\bq;w)$ are formulated as follows. 

\begin{defn}\label{defn:HP}
\begin{enumerate}
\item 
Let $\frH_s$ denote the $\Q$-vector space freely generated by all tuples $\bs=(s_1,\ldots,s_d)$ 
of complex numbers $s_i$: 
\[\frH_s\coloneqq\bigoplus_{d\ge 0}\bigoplus_{\bs\in\C^d}\Q\ \bs.\] 
The \emph{harmonic product} $*$ is a bilinear product on $\frH_s$ defined inductively by 
\begin{gather*}
\varnothing*\bs=\bs*\varnothing\coloneqq\bs, \\
\bt*\bs\coloneqq (t_1,\bt^{[1]}*\bs)+(s_1,\bt*\bs^{[1]})+(t_1+s_1,\bt^{[1]}*\bs^{[1]}), 
\end{gather*}
where $\bt=(t_1,\ldots,t_c)$ and $\bs=(s_1,\ldots,s_d)$ (recall that $\bs^{[1]}$ stands for $(s_2,\ldots,s_d)$). 
This equips $\frH_s$ with a commutative $\Q$-algebra structure, with the empty tuple $\varnothing$ 
being the unit element (cf.~\cite{Ho}).

\item 
Let $\frH_q$ denote the $\Q$-vector space freely generated by all tuples $\bq=(q_1,\ldots,q_d)$ 
of positive real numbers $q_i$: 
\[\frH_q\coloneqq\bigoplus_{d\ge 0}\bigoplus_{\bq\in(\R_{>0})^d}\Q\ \bq. \]
The \emph{harmonic product} $\astb$ is a bilinear product on $\frH_q$ defined inductively by 
\begin{gather*}
\varnothing\astb\bq=\bq\astb\varnothing\coloneqq\bq, \\
\bp\astb\bq\coloneqq (p_1,\bp^{[1]}\astb\bq)+(q_1,\bp\astb\bq^{[1]})+(p_1q_1,\bp^{[1]}\astb\bq^{[1]}), 
\end{gather*}
where $\bp=(p_1,\ldots,p_c)$ and $\bq=(q_1,\ldots,q_d)$.  
This equips $\frH_q$ with a commutative $\Q$-algebra structure, with the empty tuple $\varnothing$ 
being the unit element. 

\item 
We extend the maps $\bs\mapsto \Psi(\bs;w)$ and $\bq\mapsto G(\bq;w)$ to linear maps 
\[\Psi(-;w)\colon\frH_s\to\C, \qquad G(-;w)\colon\frH_q\to\C,\] 
respectively. 
\end{enumerate}
\end{defn}

\begin{thm}\label{thm:HP}
The linear maps $\Psi(-;w)\colon\frH_s\to\C$ and $G(-;w)\colon\frH_q\to\C$ are $\Q$-algebra homomorphisms 
with respect to the harmonic product. In other words, the harmonic relations  
\[\Psi(\bt;w)\Psi(\bs;w)=\Psi(\bt*\bs;w),\qquad G(\bp;w)G(\bq;w)=G(\bp\astb\bq;w)\]
hold. 
\end{thm}

We comment on several possible methods to prove these harmonic relations: 
\begin{enumerate}
\item 
One can show the harmonic relation of $G(-;w)$ directly (not shortly, though) from the definition. 
Then it is possible to deduce the harmonic relation of $\Psi(-;w)$ 
by the integral expression \eqref{eq:Psi integral}, 
just as in the case of \eqref{eq:HP Psi example} and \eqref{eq:HP G example}. 

\item 
Conversely, the harmonic relation of $\Psi(-;w)$ implies that of $G(-;w)$ via the inverse Mellin transform. 
To show the former, one might use the same argument as Kawashima's proof of the harmonic relation 
of Kawashima functions \cite[Theorem 5.3]{Ka}: 
if the interpolation $\Psi(\bs;w)$ of the values at non-negative integers $w=N$ is unique in some sense, 
then the relation of $\Psi(-;w)$ follows from that of $\Psi(\bs;N)=\zeta_N(\bs)$, which is evident. 
In Kawashima's proof, the uniqueness of the interpolation was provided by the theory of Newton series interpolation. 

\item 
By the expression \eqref{eq:Psi Hurwitz}, one can deduce the harmonic relation of $\Psi(-;w)$ 
from those of $\zeta^\star(-;w)$ and $\zeta(-)$. 
The harmonic relation of $G(-;w)$ is obtained in a parallel way via the expression \eqref{eq:G^d series}. 
In fact, this implication is purely algebraic once a Hopf algebra structure is recognized. 
\end{enumerate}
In what follows, we explain the third method in more detail. 

Let $(X,\,\cdot\,)$ be an abelian semigroup with an operation $^^ ^^ \cdot"$.
Let $\frH=\frH(X)$ denote the $\Q$-vector space freely generated by all tuples $\bx=(x_1,\ldots,x_d)$ 
of elements $x_i\in X$: 
\[\frH(X)\coloneqq\bigoplus_{d\ge 0}\bigoplus_{\bx\in X^d}\Q\ \bx.\] 
$\frH(X)=(\frH(X),*,\eta,\Delta,\varepsilon,R)$ will be a commutative Hopf algebra over $\Q$ 
with following operators (see, e.g., \cite{Ho2}): 
\begin{align*}
* &\colon \frH\otimes\frH \longrightarrow \frH \quad(\text{product})\\
\intertext{
defined inductively by 
\[
\varnothing*\bx=\bx*\varnothing\coloneqq\bx, \quad
\bx*\by\coloneqq (x_1,\bx^{[1]}*\by)+(y_1,\bx*\by^{[1]})+(x_1y_1,\bx^{[1]}*\by^{[1]}), 
\]
}
\eta &\colon \Q\longrightarrow\frH, \quad\eta(1)=\emptyset \quad (\text{unit}), \\
\Delta &\colon \frH\longrightarrow\frH \otimes\frH, \quad\Delta(\bx)=\sum_{l=0}^d(x_1,\dots,x_l)\otimes(x_{l+1},\dots,x_d) \quad(\text{coproduct}),\\
\varepsilon &\colon \frH\longrightarrow \Q, \quad\varepsilon(\bx)=\begin{cases}0 &(\bx\ne \emptyset),\\ 1&(\bx=\emptyset),\end{cases} \quad(\text{counit}),\\
R&\colon\frH\longrightarrow \frH, \quad R(\bx)=(-1)^d(x_d,\dots,x_1)^{\star} \quad(\text{antipode})
\end{align*}
where 
\[
(x_1,\dots,x_d)^{\star}=\sum_{\square=\,,\,\text{or}\,\cdot}(x_1\,\square\,\dots\,\square\,x_d). 
\]

Let $\Hom(\frH(X),\C)$ be the $\Q$-vector space of all $\Q$-linear maps from $\frH(X)$ to $\C$.  
For $F_1, F_2\in\Hom(\frH(X),\C)$, let $F_1\odot F_2\in\Hom(\frH(X),\C)$ be a convolution product of $F_1$ and $F_2$ defined by 
\[
F_1\odot F_2= m(F_1\otimes F_2)\Delta\colon\frH\overset{\Delta}{\longrightarrow}\frH\otimes\frH
\xrightarrow{F_1\otimes F_2}\C\otimes\C\overset{m}{\longrightarrow}\C
\]
with the usual product $m$ in $\C$.

When the semigroup $(X,\,\cdot\,)$ is the additive group $(\C, \,+\,)$ or 
the multiplicative group $(\R_{>0},\,\cdot\ )$, we have $\frH(X)=\frH_s$ or $\frH_q$, respectively.
Actually, we use Hopf subalgebras of them defined as follows: 
\begin{align*}
&\frH^\circ_s\coloneqq \frH\bigl(\{s\in\C\mid\Re(s)>1\},+\bigr)\subset \frH_s, \\
&\frH^\circ_q\coloneqq \frH\bigl(\{q\in\R^\times\mid 0<q<1\},\,\cdot\,\bigr)\subset \frH_q. 
\end{align*}
Indeed, by analyticity, it is sufficient to prove Theorem \ref{thm:HP} 
on these subalgebras.

For $w\in\C\setminus\R_{\le -1}$, let $H_w\colon\frH^\circ_q\longrightarrow\C$ 
and $Z_w\colon\frH^\circ_s\longrightarrow\C$ 
be the $\Q$-linear maps defined by 
\begin{alignat*}{2}
H_w(q_1,\dots,q_d)&\coloneqq\sum_{j_1>\dots>j_d\geq 0}q_1^{j_1+w}\dots q_d^{j_d+w}
& &(0<q_i<1),\\
Z_w(s_1,\dots,s_d)&\coloneqq\sum_{j_1>\dots>j_d> 0}\dfrac{1}{(w+j_1)^{s_1}\dots (w+j_d)^{s_d}}
&\quad &(s_i\in\C,\,\Re(s_i)>1).  
\end{alignat*}

\begin{prop}
The linear maps $G(-;w)$ and $\Psi(-;w)$ are represented by using  
convolution product, that is, 
\[
G(-;w)=(H_w R)\odot H_0, 
\quad
\Psi(-;w)=(Z_w R)\odot Z_0.
\]
\end{prop}

\begin{proof}
These equalities hold immediately from  \eqref{eq:G^d series} and \eqref{eq:Psi Hurwitz}.
\end{proof} 

\begin{proof}[Proof of Theorem \ref{thm:HP}] 
It is easily verified that the map $H_w\colon\frH^\circ_q\longrightarrow\C$ is a $\Q$-algebra homomorphism, 
and the antipode $R$ is also a $\Q$-algebra automorphism of $\frH^\circ_q$ (since $*$ is commutative). 
Since the set $\Alg(\frH(X),\C)$ of $\Q$-algebra homomorphisms from $\frH(X)$ to $\C$ 
becomes a group under the convolution product $\odot$ (cf.\ Theorem 4.0.5 of \cite{sweedler}), 
$G(-;w)=(H_w R)\odot H_0\in\Alg(\frH^\circ_q,\C)$ holds. 
Similarly, we obtain $\Psi(-;w)=(Z_w R)\odot Z_0\in\Alg(\frH^\circ_s,\C)$. 
Thus, we have Theorem \ref{thm:HP}. 
\end{proof}

\section{Further properties of $G^d(\bq;w)$}
In this section, we study several properties of $G^d({\bq};w)$, 
for example, the Newton series expansion, special values at ${\bq}=(1,\dots,1)$, 
and additivity and multiplicativity on $w$.  

\subsection{Newton series expansion}
We introduce the following expression of $G^d({\bq};w)$ as the Newton series:
\begin{prop}
For $d\ge 1$ and $\bq=(q_1,\ldots,q_d)\in U_d$ with $\abs{1-q_1\cdots q_l}<1$  ($l=1,\dots,d$), it holds 
\[
G^d({\bq};w)=\left(\prod_{l=1}^{d-1}q_l^{d-l}\right)
\sum_{j=0}^{\infty}
\binom{w}{d+j}(-1)^j
\sum_{\substack{i_1,\ldots,i_d\ge 0\\ i_1+\dots+i_d=j}}
\prod_{l=1}^{d}(1-q_1\dots q_{l})^{i_{l}}.
\]
\end{prop}
\begin{proof}
Since $q^w=(1-(1-q))^w=\sum_{j=0}^{\infty}\binom{w}{j}(-1)^j(1-q)^j$ 
by the binomial expansion, $1-q^w$ can be represented as  
\[
1-q^w=(1-q)\sum_{j=0}^{\infty}\binom{w}{j+1}(-1)^j(1-q)^j.
\]
This gives the claim for $d=1$. 
Then we proceed by induction on $d$. 
Substituting the induction hypothesis for $G^{d-1}$ 
into the definition \eqref{eq:G^d} of $G^d$, we obtain 
\begin{align*}
G^d(\bq;w)=
\left(\prod_{l=1}^{d-2}q_l^{d-1-l}\right)
\sum_{j=0}^\infty\binom{w}{d-1+j}(-1)^j
\sum_{\substack{i_1,\ldots,i_{d-1}\ge 0\\ i_1+\cdots+i_{d-1}=j}}
\prod_{l=1}^{d-2}(1-q_1\cdots q_l)^{i_l}&\\
\times\frac{(1-q_1\cdots q_{d-1})^{i_{d-1}}-(1-q_1\cdots q_d)^{i_{d-1}}}{1-q_d}.&
\end{align*}
Thus the claim for $G^d$ follows from the identity 
\begin{align*}
&\frac{(1-q_1\cdots q_{d-1})^{i_{d-1}}-(1-q_1\cdots q_d)^{i_{d-1}}}{1-q_d}\\
&=-q_1\cdots q_{d-1}\sum_{\substack{i'_{d-1},i'_d\ge 0\\ i'_{d-1}+i'_d=i_{d-1}-1}}
(1-q_1\cdots q_{d-1})^{i'_{d-1}}(1-q_1\cdots q_d)^{i'_d}. \qedhere
\end{align*}
\end{proof}

\subsection{Special values of $G^d({\bq};w)$ at some ${\bq}$} 
\begin{prop} 
For $\Re(w)>0$, we have
\[
\lim_{(q_1, \ldots,q_d)\to (0,\ldots, 0)} G^d(q_1,\ldots, q_d;w) 
=\begin{cases}
1& (d=1),\\
0& (d\ge 2). 
\end{cases}
\]
\end{prop}
\begin{proof}
This follows from \eqref{eq:G^d explicit}. 
\end{proof}

\begin{prop}\label{binomial}
The following hold: 
\begin{align}
\label{lim11}
G^d(\underbrace{1,\ldots,1}_{d};w)
&=\binom{w}{d}, \\
\label{expl1q}
G^d(\underbrace{1,\ldots,1}_{d-1},q;w)
&=\frac{1}{(q-1)^d}\left(q^w-\sum_{l=0}^{d-1}\binom{w}{l}(q-1)^{l}\right). 
\end{align}
Here we assume that $q\ne 1$ in \eqref{expl1q}. 
\end{prop}
\begin{proof}

We prove \eqref{lim11} and \eqref{expl1q} simultaneously by induction on $d$. 
The case $d=1$ is obvious. For $d\geq 2$, by the definition of $G^d$ and the induction hypothesis, 
we have 
\begin{align*}
G^d(\underbrace{1,\ldots,1}_{d-1},q;w)
&=\frac{1}{1-q}\Biggl\{G^{d-1}(\underbrace{1,\ldots,1}_{d-1};w)
-G^{d-1}(\underbrace{1,\ldots,1}_{d-2},q;w)\Biggr\}\\
&=\frac{1}{1-q}\Biggl\{\binom{w}{d-1}-\frac{1}{(q-1)^{d-1}}
\biggl(q^w-\sum_{l=0}^{d-2}\binom{w}{l}(q-1)^l\biggr)\Biggr\}\\
&=\frac{1}{(q-1)^d}\Biggl(q^w-\sum_{l=0}^{d-1}\binom{w}{l}(q-1)^{l}\Biggr), 
\end{align*}
which shows \eqref{expl1q} for $G^d$. Moreover, by letting $q\to 1$, we obtain 
\[G^d(\underbrace{1,\ldots,1}_{d};w)
=\lim_{q\to 1}\frac{1}{(q-1)^d}
\Biggl(\sum_{l=0}^\infty\binom{w}{l}(q-1)^{l}-\sum_{l=0}^{d-1}\binom{w}{l}(q-1)^{l}\Biggr)
=\binom{w}{d}.\]
Thus \eqref{lim11} holds for $G^d$. 
\end{proof}

\begin{cor}
The generating function  of $G^d(\underbrace{1,\ldots,1}_{d-1},q;w)$ is given by
\begin{align*}
\sum_{d=1}^\infty G^d(\underbrace{1,\ldots,1}_{d-1},q;w)X^d=\frac{((1+X)^w-q^w)X}{1-q+X}. 
\end{align*}
\end{cor}
\begin{proof}
By \eqref{expl1q}, we have 
\[\sum_{d=1}^\infty G^d(\underbrace{1,\ldots,1}_{d-1},q;w)X^d
=\sum_{d=1}^\infty\biggl(\frac{X}{q-1}\biggr)^d\Biggl(q^w-\sum_{l=0}^{d-1}\binom{w}{l}(q-1)^l\Biggr). \]
Then the claim follows from the identities 
\[\sum_{d=1}^\infty\biggl(\frac{X}{q-1}\biggr)^d q^w=\frac{q^w X}{q-1-X}\]
and 
\begin{align*}
\sum_{d=1}^\infty\biggl(\frac{X}{q-1}\biggr)^d\sum_{l=0}^{d-1}\binom{w}{l}(q-1)^l
&=\sum_{l=0}^\infty\binom{w}{l}(q-1)^l\sum_{d=l+1}^\infty \biggl(\frac{X}{q-1}\biggr)^d\\
&=\sum_{l=0}^\infty\binom{w}{l}\frac{X^{l+1}}{q-1-X}=\frac{(1+X)^w X}{q-1-X}. \qedhere
\end{align*}
\end{proof}

\begin{prop}
Assume that $\bq=(q_1,\ldots,q_{d-1})\in U_{d-1}$ satisfies that $q_m\cdots q_{d-1}\ne 1$ 
for $m=1,\ldots,d-1$. Then 
\begin{align}
\label{expdif}
G^d(\bq,1;w)
&=\dfrac{(-1)^{d-1}w\prod_{m=1}^{d-1}q_m^w}{\prod_{m=1}^{d-1}(1-\prod_{k=m}^{d-1}q_k)} \\
\notag&\quad
+\sum_{l=1}^{d-1}\dfrac{(-1)^{d-1-l}\left(\prod_{m=l}^{d-1}q_m\right) 
G^l({\bq}_{[l-1]},\prod_{m=l}^{d-1}q_m;w)}
{\prod_{m=l}^{d-1}(1-\prod_{k=m}^{d-1}q_k)}. 
\end{align}
\end{prop}
\begin{proof}
By definition \eqref{eq:G^d}, it holds that 
\[G^d(\bq,1;w)=\biggl(\frac{\partial}{\partial q_d}
G^{d-1}(\bq_{[d-2]},q_{d-1}q_d;w)\biggr)\bigg|_{q_d=1}. \]
We also have 
\begin{align*}
\biggl(\frac{\partial}{\partial q_d}
&G^l(\bq_{[l-1]},q_{l}\cdots q_{d-1}q_d;w)\biggr)\bigg|_{q_d=1}\\
&=\frac{q_{l}\cdots q_{d-1}}{1-q_{l}\cdots q_{d-1}}G^{l}(\bq_{[l-1]},q_{l}\cdots q_{d-1};w)\\
&\quad -\frac{1}{1-q_{l}\cdots q_{d-1}}
\biggl(\frac{\partial}{\partial q_d}
G^{l-1}(\bq_{[l-2]},q_{l-1}\cdots q_{d-1}q_d;w)\biggr)\bigg|_{q_d=1}
\end{align*}
for $l=2,\ldots,d-1$ and 
\begin{align*}
\biggl(\frac{\partial}{\partial q_d}
&G^1(q_1\cdots q_{d-1}q_d;w)\biggr)\bigg|_{q_d=1}\\
&=\frac{q_1\cdots q_{d-1}}{1-q_1\cdots q_{d-1}}G^1(q_1\cdots q_{d-1};w)
-\frac{w(q_1\cdots q_{d-1})^w}{1-q_1\cdots q_{d-1}}. 
\end{align*}
By combining these together, the formula \eqref{expdif} follows.  
\end{proof}

\subsection{Property of $G^d({\bq};w)$ as a function in $w$} 
We introduce the additivity and multiplicativity of $G^d({\bq};w)$ with respect to $w$. 

\begin{prop}\label{ad-prod}
Let $u, v\in\C$ and $\bq=(q_1,\ldots,q_d)\in U_d$, $d\ge 1$. 
\begin{enumerate}
\item 
It holds that 
\begin{align}
\label{additivity}
&G^{d}({\bq};u+v)\\
\notag&=G^d({\bq};u)+q_1^u G^d({\bq};v)
+q_1^u\sum_{l=1}^{d-1}G^{d-l}\left(\prod_{m=2}^{l+1}q_m,{\bq}^{[l+1]};u\right)
G^l(q_1q_2,{\bq}_{[l+1]}^{[2]};v), 
\end{align}
where ${\bq}_{[l]}^{[m]}\coloneqq({\bq}_{[l]})^{[m]}=(q_{m+1},\dots,q_l)$ for $0\leq m\leq l\leq d$. 
\item 
If $(q_1^u,\ldots,q_d^u)\in U_d$, it holds that 
\begin{align}
\label{multiplicativity}
&G^d({\bq};uv)\\
\notag 
&=\sum_{l=1}^d
\sum_{d=j_0\ge j_1\ge \cdots\ge j_l=l}
\prod_{n=1}^l
G^{j_{n-1}-j_n+1}
\left(\prod_{m=n}^{j_n} q_m,{\bq}_{[j_{n-1}]}^{[j_n]};u\right) 
G^l(q_1^u,\dots,q_l^u;v). 
\end{align}
\end{enumerate}
\end{prop}

\begin{proof}
(1) Let $F^d(\bq;u,v)$ denote the right-hand side of \eqref{additivity}. 
For $d=1$, we have 
\[F^1(q;u,v)=\frac{1-q^u}{1-q}+q^u\frac{1-q^v}{1-q}=\frac{1-q^{u+v}}{1-q}=G^1(q;u+v). \]
Thus it is sufficient to show that $F^d(\bq;u,v)$ satisfies the same recursive relation as $G^d(\bq;u+v)$, 
that is, it holds that 
\[(1-q_d)F^d(\bq;u,v)=F^{d-1}(\bq_{[d-1]};u,v)-F^{d-1}(\tilde{\bq}_{[d-1]};u,v)\]
for $d\ge 2$, where we set $\tilde{\bq}_{[d-1]}\coloneqq (q_1,\ldots,q_{d-2},q_{d-1}q_d)$. 
This is easily checked for $d=2$, and for $d\ge 3$, 
this follows from the identities below: 
\begin{align*}
&(1-q_d)G^d(\bq;u)=G^{d-1}(\bq_{[d-1]};u)-G^{d-1}(\tilde{\bq}_{[d-1]};u), \\
&(1-q_d)G^d(\bq;v)=G^{d-1}(\bq_{[d-1]};v)-G^{d-1}(\tilde{\bq}_{[d-1]};v), \\ 
&(1-q_d)G^{d-l}(q_2\cdots q_{l+1},\bq^{[l+1]};u)G^l(q_1q_2,\bq^{[2]}_{[l+1]};v)\\
&\qquad=G^{d-l-1}(q_2\cdots q_{l+1},\bq^{[l+1]}_{[d-1]};u)G^l(q_1q_2,\bq^{[2]}_{[l+1]};v)\\
&\qquad\quad -G^{d-l-1}(q_2\cdots q_{l+1},\tilde{\bq}^{[l+1]}_{[d-1]};u)G^l(q_1q_2,\bq^{[2]}_{[l+1]};v)
\end{align*}
for $1\le l\le d-3$, and 
\begin{align*}
(1-q_d)&\Bigl(G^2(q_2\cdots q_{d-1},q_d;u)G^{d-2}(q_1q_2,\bq^{[2]}_{[d-1]};v)\\
&\quad +G^1(q_2\cdots q_d;u)G^{d-1}(q_1q_2,\bq^{[2]}_{[d]};v)\Bigr)\\
&=\Bigl\{G^1(q_2\cdots q_{d-1};u)-G^1(q_2\cdots q_d;u)\Bigr\}G^{d-2}(q_1q_2,\bq_{[d-1]}^{[2]};v)\\
&\quad+G^1(q_2\cdots q_d;u)\Bigl\{G^{d-2}(q_1q_2,\bq^{[2]}_{[d-1]};v)
-G^{d-2}(q_1q_2,\tilde{\bq}^{[2]}_{[d-1]};v)\Bigr\}\\
&=G^1(q_2\cdots q_{d-1};u)G^{d-2}(q_1q_2,\bq^{[2]}_{[d-1]};v)\\
&\quad-G^1(q_2\cdots q_d;u)G^{d-2}(q_1q_2,\tilde{\bq}^{[2]}_{[d-1]};v). 
\end{align*}

(2) Let $H^d(\bq;u,v)$ denote the right-hand side of \eqref{multiplicativity}. 
For $d=1$, it holds 
\[H^1(q;u,v)=\frac{1-q^u}{1-q}\frac{1-(q^u)^v}{1-q^u}=\frac{1-q^{uv}}{1-q}=G^1(q;uv). \]
Thus, again, it is sufficient to show the relation 
\begin{equation}\label{eq:H^d recursion}
(1-q_d)H^d(\bq;u,v)=H^{d-1}(\bq_{[d-1]};u,v)-H^{d-1}(\tilde{\bq}_{[d-1]};u,v)
\end{equation}
for $d\ge 2$. In fact, we will prove a refinement of this identity. 

Let $J_d$ denote the set of sequences $\bj=(j_0,j_1,\ldots,j_l)$ satisfying $d=j_0\ge j_1\ge\cdots\ge j_l=l$. 
For such a sequence $\bj$, let $H^d(\bq;u,v)_\bj$ be the corresponding term in $H^d(\bq;u,v)$, 
that is, 
\[
H^d(\bq;u,v)_\bj\coloneqq \prod_{n=1}^l
G^{j_{n-1}-j_n+1}
\left(\prod_{m=n}^{j_n} q_m,{\bq}_{[j_{n-1}]}^{[j_n]};u\right) 
G^l(q_1^u,\dots,q_l^u;v). 
\]
For each sequence $\bi=(i_0,i_1,\ldots,i_l)\in J_{d-1}$, we define a subset $J_d(\bi)$ of $J_d$ as follows. 
When $l<d-1$, let $t$ be the minimum number with $i_t<d-1$, so that 
$\bi=(\underbrace{d-1,\ldots,d-1}_{t},i_t,\ldots,i_l)$. Then we set 
\[J_d(\bi)\coloneqq 
\Bigl\{(\underbrace{d,\ldots,d}_{s},\underbrace{d-1,\ldots,d-1}_{t-s},i_t,\ldots,i_l)
\Bigm| 1\le s\le t\Bigr\}. \]
When $l=d-1$, we set 
\begin{align*}
J_d(\underbrace{d-1,\ldots,d-1}_{d})&\coloneqq 
\Bigl\{(\underbrace{d,\ldots,d}_{s},\underbrace{d-1,\ldots,d-1}_{d-s})
\Bigm| 1\le s\le d-1\Bigr\}
\cup\Bigl\{(\underbrace{d,\ldots,d}_{d+1})\Bigr\}. 
\end{align*}
Then, since $J_d$ is the disjoint union of $J_d(\bi)$ where $\bi$ runs over $J_{d-1}$, 
the equality \eqref{eq:H^d recursion} follows from the identity 
\begin{equation}\label{eq:H^d bi}
\sum_{\bj\in J_d(\bi)}(1-q_d)H^d(\bq;u,v)_\bj=H^{d-1}(\bq_{[d-1]};u,v)_\bi
-H^{d-1}(\tilde{\bq}_{[d-1]};u,v)_\bi 
\end{equation}
for $\bi\in J_{d-1}$, which will be shown below. 

First we consider $\bi=(i_0,i_1,\ldots,i_l)$ with $l<d-1$. 
Let $t$ be the minimum number with $i_t<d-1$. 
For $\bj=(\underbrace{d,\ldots,d}_{s},\underbrace{d-1,\ldots,d-1}_{t-s},i_t,\ldots,i_l)\in J_d(\bi)$ 
with $1\le s\le t-1$, we have 
\begin{align*}
&(1-q_d)H^d(\bq;u,v)_{\bj}\\
&=(1-q_d)\prod_{n=1}^{s-1}G^1(q_n\cdots q_d;u)\cdot G^{2}(q_s\cdots q_{d-1},q_d;u) 
\cdot \prod_{n=s+1}^{t-1}G^1(q_n\cdots q_{d-1};u)\\
&\quad\times G^{d-i_t}(q_t\cdots q_{i_t},\bq^{[i_t]}_{[d-1]};u)
\cdot\prod_{n=t+1}^l G^{i_{n-1}-i_n+1}(q_n\cdots q_{i_n},\bq^{[i_n]}_{[i_{n-1}]};u)
\cdot G^l(q_1^u,\ldots,q_l^u;v)\\
&=\prod_{n=1}^{s-1}G^1(q_n\cdots q_d;u)
\cdot\Bigl\{G^1(q_s\cdots q_{d-1};u)-G^1(q_s\cdots q_d;u)\Bigr\}
\cdot\prod_{n=s+1}^{t-1}G^1(q_n\cdots q_{d-1};u)\\
&\quad\times G^{d-i_t}(q_t\cdots q_{i_t},\bq^{[i_t]}_{[d-1]};u)
\cdot\prod_{n=t+1}^l G^{i_{n-1}-i_n+1}(q_n\cdots q_{i_n},\bq^{[i_n]}_{[i_{n-1}]};u)
\cdot G^l(q_1^u,\ldots,q_l^u;v). 
\end{align*}
On the other hand, for $\bj=(\underbrace{d,\ldots,d}_{t},i_t,\ldots,i_l)$, we have 
\begin{align*}
&(1-q_d)H^d(\bq;u,v)_{\bj(t)}\\
&=(1-q_d)\prod_{n=1}^{t-1}G^1(q_n\cdots q_d;u) 
\cdot G^{d-i_t+1}(q_t\cdots q_{i_t},\bq^{[i_t]};u)\\
&\quad\times\prod_{n=t+1}^l G^{i_{n-1}-i_n+1}(q_n\cdots q_{i_n},\bq^{[i_n]}_{[i_{n-1}]};u)
\cdot G^l(q_1^u,\ldots,q_l^u;v)\\
&=\prod_{n=1}^{t-1}G^1(q_n\cdots q_d;u)
\cdot\Bigl\{G^{d-i_t}(q_t\cdots q_{i_t},\bq^{[i_t]}_{[d-1]};u)
-G^{d-i_t}(q_t\cdots q_{i_t},\tilde{\bq}^{[i_t]}_{[d-1]};u)\Bigr\}\\
&\quad\times\prod_{n=t+1}^l G^{i_{n-1}-i_n+1}(q_n\cdots q_{i_n},\bq^{[i_n]}_{[i_{n-1}]};u)
\cdot G^l(q_1^u,\ldots,q_l^u;v). 
\end{align*}
Summing these terms together, we obtain 
\begin{align*}
\sum_{\bj\in J_d(\bi)}(1-q_d)H^d(\bq;u,v)_\bj
&=\Biggl\{\prod_{n=1}^{t-1}G^1(q_n\cdots q_{d-1};u)\cdot 
G^{d-i_t}(q_t\cdots q_{i_t},\bq^{[i_t]}_{[d-1]};u)\\
&\qquad-\prod_{n=1}^{t-1}G^1(q_n\cdots q_{d};u)\cdot 
G^{d-i_t}(q_t\cdots q_{i_t},\tilde{\bq}^{[i_t]}_{[d-1]};u)\Biggr\}\\
&\quad\times\prod_{n=t+1}^l G^{i_{n-1}-i_n+1}(q_n\cdots q_{i_n},\bq^{[i_n]}_{[i_{n-1}]};u) 
\cdot G^l(q_1^u,\ldots,q_l^u;v)\\
&=H^{d-1}(\bq_{[d-1]};u,v)_\bi-H^{d-1}(\tilde{\bq}_{[d-1]};u,v)_\bi, 
\end{align*}
which establishes \eqref{eq:H^d bi} in the present case. 

Next we consider $\bi=(\underbrace{d-1,\ldots,d-1}_{d})\in J_{d-1}$. 
For $\bj=(\underbrace{d,\ldots,d}_{s},\underbrace{d-1,\ldots,d-1}_{d-s})$ with $1\le s\le d-1$, 
we have 
\begin{align*}
(1-q_d)H^d(\bq;u,v)_{\bj}
&=(1-q_d)\prod_{n=1}^{s-1}G^1(q_n\cdots q_d;u)\cdot G^2(q_s\cdots q_{d-1},q_d;u)\\
&\quad\times \prod_{n=s+1}^{d-1}G^1(q_n\cdots q_{d-1};u)\cdot G^{d-1}(q_1^u,\ldots,q_{d-1}^u;v)\\
&=\prod_{n=1}^{s-1}G^1(q_n\cdots q_d;u)\cdot 
\Bigl\{G^1(q_s\cdots q_{d-1};u)-G^1(q_s\cdots q_d;u)\Bigr\}\\
&\quad\times \prod_{n=s+1}^{d-1}G^1(q_n\cdots q_{d-1};u)\cdot G^{d-1}(q_1^u,\ldots,q_{d-1}^u;v). 
\end{align*}
On the other hand, for $\bj=(\underbrace{d,\ldots,d}_{d+1})$, we have 
\begin{align*}
&(1-q_d)H^d(\bq;u,v)_{\bj}\\
&=(1-q_d)\prod_{n=1}^{d}G^1(q_n\cdots q_d;u)\cdot G^d(q_1^u,\ldots,q_d^u;v)\\
&=\prod_{n=1}^{d-1}G^1(q_n\cdots q_d;u)\cdot(1-q_d^u) G^d(q_1^u,\ldots,q_d^u;v)\\
&=\prod_{n=1}^{d-1}G^1(q_n\cdots q_d;u)
\cdot\Bigl\{G^{d-1}(q_1^u,\ldots,q_{d-1}^u;v)-G^{d-1}(q_1^u,\ldots,(q_{d-1}q_d)^u;v)\Bigr\}. 
\end{align*}
Summing these terms together, we obtain 
\begin{align*}
\sum_{\bj\in J_d(\bi)}(1-q_d)H^d(\bq;u,v)_\bj
&=\prod_{n=1}^{d-1}G^1(q_n\cdots q_{d-1};u)\cdot G^{d-1}(q_1^u,\ldots,q_{d-1}^u;v)\\
&\quad-\prod_{n=1}^{d-1}G^1(q_n\cdots q_d;u)\cdot G^{d-1}(q_1^u,\ldots,(q_{d-1}q_d)^u;v)\\
&=H^{d-1}(\bq_{[d-1]};u,v)_\bi-H^{d-1}(\tilde{\bq}_{[d-1]};u,v)_\bi
\end{align*}
as desired. 
Thus we have proven \eqref{eq:H^d bi} for all $\bi\in J_{d-1}$, and hence \eqref{eq:H^d recursion}. 
\end{proof}

As a consequence of Proposition \ref{ad-prod}, 
the right-hand sides of \eqref{additivity} and \eqref{multiplicativity} 
turn out to be symmetric with respect to $u$ and $v$ 
(for the latter, we also need to assume $(q_1^v,\ldots,q_d^v)\in U_d$). 

\begin{ex}
In the case of $d=2$, 
\begin{align*}
G^2(q_1,q_2;u+v)&=G^2(q_1,q_2;u)+q_1^uG^2(q_1,q_2;v)+q_1^uG^1(q_2;u)G^1(q_1q_2;v)\\
&=G^2(q_1,q_2;v)+q_1^vG^2(q_1,q_2;u)+q_1^vG^1(q_2;v)G^1(q_1q_2;u).
\end{align*}
\end{ex}

\begin{ex}
In the case of $d=1$, 
\begin{align*}
G^1(q;uv)&=\frac{1-q^u}{1-q}\frac{1-q^{uv}}{1-q^u}=G^1(q;u)G^1(q^u;v)\\
&=\frac{1-q^v}{1-q}\frac{1-q^{uv}}{1-q^v}=G^1(q;v)G^1(q^v;u)
\end{align*}
hold. 
In cases of $d=2$ and $d=3$, one has 
\begin{align*}
&G^2(q_1,q_2;uv)=G^2(q_1,q_2;u)G^1(q_1^u;v)+G^1(q_1q_2;u)G^1(q_2;u)G^2(q_1^u,q_2^u;v),\\
&G^3(q_1,q_2,q_3;uv)=G^3(q_1,q_2,q_3;u)G^1(q_1^u;v)\\
&\qquad+\left(G^1(q_1q_2;u)G^2(q_2,q_3;u)+G^2(q_1q_2,q_3;u)G^1(q_2q_3;u)\right)G^2(q_1^u,q_2^u;v)\\
&\qquad+G^1(q_1q_2q_3;u)G^1(q_2q_3;u)G^1(q_3;u)G^3(q_1^u,q_2^u,q_3^u;v).
\end{align*}
\end{ex}

\begin{ex}
By \eqref{additivity} with $v=1$, we have the following difference formula: 
\[
G^d({\bq};u+1)-G^d({\bq};u)=q_1^u G^{d-1}({\bq}^{[1]};u). 
\]
Note that its Mellin transform yields the difference equation \eqref{eq:Psi difference} of $\Psi$. 
\end{ex}

\subsection{Applications to binomial coefficients}

The identity \eqref{lim11} allows us to deduce some formulas for binomial coefficients 
from corresponding ones for the function $G(\bq;w)$. 
We show the following Corollaries \ref{cor:binomial1} and \ref{cor:binomial2} 
as examples of this procedure. 
Though these formulas for binomial coefficients can be directly proven, 
it may be of some interest to notice such a relationship with our function $G(\bq;w)$.

\begin{cor}\label{cor:binomial1} 
For $u, v\in\C$ and $d\in\Z_{\geq 0}$, we have 
\[
\binom{u+v}{d}=\sum_{j=0}^d\binom{v}{j}\binom{u}{d-j}
\]
and
\[
\binom{uv}{d}
=\sum_{l=0}^d
\left(\sum_{\substack{i_1,\ldots,i_l\ge 1\\ i_1+\cdots+i_l=d}}
\binom{u}{i_1}\binom{u}{i_2}\cdots\binom{u}{i_l}\right)\binom{v}{l}. 
\]
\end{cor} 
\begin{proof} 
Putting ${\bq}=(1,\dots,1)$ in formulas in Proposition \ref{ad-prod}, we obtain the claim. 
\end{proof}

For $c,d\geq 0$, 
let ${\bp}$ and ${\bq}$ be tuples of length $c$ and $d$ respectively. 
For $m\geq 0$, 
let $D^m(c,d)$ be the number of terms of length $m$ in the expansion of $\bp\astb\bq$. 
For example, since $(p)\astb(q)=(p,q)+(q,p)+(pq)$, we have $D^1(1,1)=1$, $D^2(1,1)=2$ and 
$D^m(1,1)=0$ for $m\ge 3$. 
Note that $D^m(c,d)$ depends only on $c$ and $d$. 

\begin{cor}\label{cor:binomial2} 
For $w\in\C$ and $c,d\in\Z_{\geq 0}$, 
\begin{equation}\label{binom-prod}
\binom{w}{c}\binom{w}{d}=\sum_{m=0}^{c+d}D^m(c,d)\binom{w}{m}
\end{equation}
holds.
\end{cor}
\begin{proof} 
If either $c$ or $d$ is zero, \eqref{binom-prod} is trivial. 
For $c,d>0$, since 
\[
(\underbrace{1,\dots,1}_{c})\astb(\underbrace{1,\dots,1}_{d})
=\sum_{m=0}^{c+d}D^m(c,d)(\underbrace{1,\dots,1}_{m}), 
\]
it follows from Theorem \ref{thm:HP} that 
\[G^c(1,\ldots,1;w)\,G^d(1,\ldots,1;w)=\sum_{m=0}^{c+d}D^m(c,d)G^m(1,\ldots,1;w). \]
Thus, by \eqref{lim11}, we have the claim. 
\end{proof}

\begin{prop} 
For $c,d,m\geq 1$, 
$D^m(c,d)$ is characterized by the recurrence formula  
\begin{equation}\label{recurrence-delanoy}
D^m(c,d)=D^{m-1}(c-1,d)+D^{m-1}(c,d-1)+D^{m-1}(c-1,d-1) 
\end{equation}
with 
\begin{equation}\label{recurrence-initial}
\left\{\begin{array}{l}
D^0(c,d)=\delta_{0c}\delta_{0d}\quad (c,d\geq 0), \\[2mm] 
D^m(c,0)=\delta_{mc}, \quad D^m(0,d)=\delta_{md}\quad (c,d,m\geq 0).
\end{array}\right. 
\end{equation}   
Equivalently, the generating function of $D^m(c,d)$ is as follows:
\begin{equation}\label{generating-fn}
\sum_{c,d,m\geq 0}D^m(c,d)x^cy^dz^m
=\frac{1}{1-xz-yz-xyz}. 
\end{equation}
\end{prop}

\begin{proof} The recurrence formula \eqref{recurrence-delanoy} and the initial condition \eqref{recurrence-initial} are directly derived by the definition of $D^m(c,d)$. 
The generating function \eqref{generating-fn} is implied easily from \eqref{recurrence-delanoy} and \eqref{recurrence-initial}. 
\end{proof}

\begin{rem}
The numbers 
\[
D(c,d)=\sum_{m=0}^{c+d} D^m(c,d)
\]
are known as Delannoy numbers, whose generating function is 
$\dfrac{1}{1-x-y-xy}$ ([A008288] in OEIS%
\footnote{The On-Line Encyclopedia of Integer Sequences, \url{https://oeis.org}}). 
\end{rem}
\begin{ex}
Since  
\[
(p_1,p_2)\astb (q)=(p_1, p_2,q)+(p_1, q,p_2)+(q,p_1,p_2)+(p_1, p_2q)+(p_1q,p_2), 
\]
we have $D^3(2,1)=3$,  $D^2(2,1)=2$,  $D^1(2,1)=D^{0}(2,1)=0$. 
Thus, by \eqref{binom-prod}, we have 
\[
\binom{w}{2}\binom{w}{1}=3\binom{w}{3}+2\binom{w}{2}. 
\]
\end{ex}

\end{document}